\documentclass[12pt]{amsart}
\usepackage{youngtab}
\usepackage{graphicx,tikz, tikz-cd}
\usepackage{caption}
\usepackage{subcaption}
\usepackage{mathrsfs,hyperref}
\usepackage{cancel}
\usepackage{fullpage}
\usepackage{framed}

\numberwithin{equation}{section}

\newtheorem{theorem}{Theorem}[section]

\newtheorem{lemma}[theorem]{Lemma}

\newtheorem{corollary}[theorem]{Corollary}

\theoremstyle{definition}

\newtheorem{example}[theorem]{Example}

\newtheorem{remark}[theorem]{Remark}
\newtheorem{algorithm}[theorem]{Algorithm}

\title{Dehn coloring and the dimer model for knots}

\author{Alexander Madaus}
\address{Department of Mathematics and Computer Science, 
Washington College, 300 Washington Avenue, Chestertown, MD
21620 U.S.A.}
\email{amadaus2@washcoll.edu}

\author{Maisie Newman}
\address{Department of Mathematics and Computer Science, Washington College, 300 Washington Avenue, Chestertown, MD 21620 U.S.A.}
\email{mnewman2@washcoll.edu}

\author{Heather M. Russell}
\address{Department of Mathematics and Computer Science, University of Richmond, U.S.A.}
\email{hrussell@richmond.edu}

\begin{document}

\begin{abstract}
Fox coloring provides a combinatorial framework for studying dihedral representations of the knot group.  The less well-known  concept of Dehn coloring captures the same data. Recent work of Carter-Silver-Williams clarifies the relationship between the two focusing on how one transitions between Fox and Dehn colorings. In our work, we relate Dehn coloring to the dimer model for knots showing that Dehn coloring data is encoded by a certain weighted balanced overlaid Tait graph. Using Kasteleyn theory, we provide graph theoretic methods for computing the determinant and Smith normal form of a knot. These constructions are closely related to Kauffman's work on a state sum  for the Alexander polynomial.
\end{abstract}

\thanks{The authors would like to thank Oliver Dasbach for his guidance and comments, the FURST program (NSF DMS \#1156273) at CSU Fresno for research support, and Jozef Przytycki for his consistent support and mentorship within the knot theory community.}

\maketitle

\section{Introduction}

An $n$-Fox coloring of a knot diagram assigns an element of $\mathbb{Z}/n$ to each {\it arc} subject to a system of congruence relations coming from the crossings. For all $n\in \mathbb{N}$, the number of $n$-Fox colorings is a knot invariant. The collection of $n$-colorings forms a $\mathbb{Z}$-module, and one can study the structure of this module to get a finer invariant. An $n$-Dehn coloring assigns an element of $\mathbb{Z}/n$ to each bounded {\it face} of a diagram such that a different system of crossing congruence conditions holds. Fox and Dehn coloring are related by a change of presentation of the knot group. 

The concept of Fox coloring was first introduced by Ralph Fox to Haverford College undergraduates as a way to the study of maps from knot groups into dihedral groups.  Fox colorings of arcs correspond to maps written in terms of the Wirtinger generators. Fox coloring is well-studied and has led to some interesting results including the Kauffman-Harary conjecture \cite{harary1999knots, livingston1993knot, mattman2009proof, przytycki19983}. Less studied are Dehn colorings \cite{carter2014three, kauffman2006formal}. Here, labeling faces of the knot diagram corresponds to maps defined on Dehn generators.

In this paper, we study knot colorings from the perspective of the dimer model for knots which is introduced in \cite{CDR}. This is closely related to Kauffman's state sum model for the Alexander polynomial \cite{kauffman2006formal} where marked states are replaced with dimer coverings. We show how Dehn coloring data can be directly obtained from the balanced overlaid Tait (BOT) graph. This is especially convenient when the knot diagram is alternating. We show that if one was to attempt the same graph theoretic interpretation of Fox coloring data, the resulting graph is not always planar. In this sense, Dehn coloring is a preferable framework.

Via Kasteleyn theory, we offer a simple proof of the classical result relating the number of maximal trees in the Tait graph of an alternating diagram to the knot determinant \cite{crowell1959nonalternating}. (This has since been generalized in the context of ribbon graphs coming from knot projections on the Turaev surface \cite{dasbach2008jones}.)  We give an algorithm for determining the Smith normal form of the knot (i.e. the structure of the coloring module) via the BOT graph. We give sufficient conditions on the BOT graph for a knot to have trivial Smith normal form. 

\section{Fox and Dehn coloring}

This section explores the relationship between Fox and Dehn $n$-colorings by building on the recent work of Carter-Silver-Williams \cite{carter2014three}. We give an explicit formula for a dihedral homomorphism coming from a Dehn coloring in Lemma \ref{reps} and show that the given map agrees with the Carter-Silver-Williams combinatorial bijection between Fox and Dehn colorings. In the exposition that follows, we fix $n\in \mathbb{N}$ and a diagram $D$ of a knot $K$. Let $A$ be the set of arcs and $F$ the set of faces of $D$.

A {\it Fox $n$-coloring} of $D$ is a map $C:A\rightarrow \mathbb{Z}/n$ with the property that twice the label of the overstrand at each crossing is congruent to the sum of the labels of the understands modulo $n$. This is illustrated in Figure \ref{coloringrels}. We denote the set of all Fox $n$-colorings of $D$ by $\textup{Col}_n(D)$.  A {\it Dehn $n$-coloring} of $D$ is a map $\widetilde{C}: F \rightarrow \mathbb{Z}/n$ for which  the sum of the face labels on either side of the over strand are congruent to one another modulo $n$ as one passes through each crossing and the unbounded face maps to 0. The local relation for Dehn coloring is shown in Figure \ref{coloringrels}. We denote the set of Dehn $n$-colorings by $\widetilde{\textup{Col}}_n(D)$.  Fox and Dehn 5-colorings of a figure-eight knot diagram can be found in Figure \ref{corrcolors}.

\begin{figure}[h]
\begin{subfigure}[b]{0.3\textwidth}
\begin{tikzpicture}[scale=.5]
\draw[style=thick] (-1.3,-1.3) to[out=45,in=225] (1.3,1.3);
\draw[style=thick] (-1.3, 1.3) to [out=315,in=135] (-.15, .15);
\draw[style=thick](.15,-.15) to [out=315, in=135](1.3,-1.3);
\node at (-1,-.4) {$a$};
\node at (-.5,1.1) {$b$};
\node at (1, -.4) {$c$};

\node at (6, 0) {$2a \equiv b+c \;(\textup{mod } n)$};
\end{tikzpicture}
\caption{Fox coloring}
\end{subfigure} \hspace{.55in}
\begin{subfigure}[b]{.3\textwidth}
\begin{tikzpicture}[scale=.5]
\draw[style=thick] (-1.3,-1.3) to[out=45,in=225] (1.3,1.3);
\draw[style=thick] (-1.3, 1.3) to [out=315,in=135] (-.15, .15);
\draw[style=thick](.15,-.15) to [out=315, in=135](1.3,-1.3);
\node at (-1.3,0) {\fbox{$a$}};
\node at (0,-1.3) {\fbox{$b$}};
\node at (0, 1.3) {\fbox{$c$}};
\node at (1.3, 0) {\fbox{$d$}};

\node at (6, 0) {$a+b \equiv c+d \;(\textup{mod } n)$};
\end{tikzpicture}
\caption{Dehn coloring}
\end{subfigure}
\caption{Coloring congruence relations}\label{coloringrels}
\end{figure}
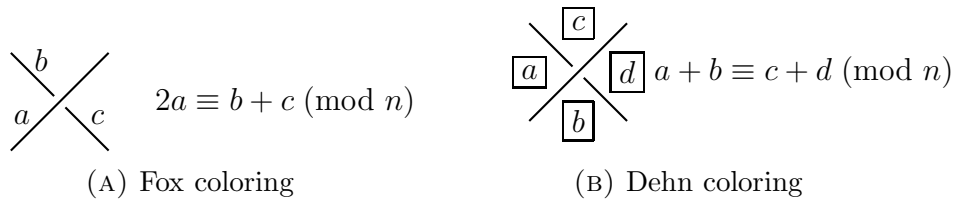

Both $\textup{Col}_n(D)$ and $\widetilde{\textup{Col}}_n(D)$ have a $\mathbb{Z}$-module structure under vector addition modulo $n$. We therefore refer to them as the Fox and Dehn $n$-coloring modules. Another fundamental property of these modules is that they are independent of choice of diagram for a knot \cite{przytycki19983}. Hence they are knot invariants, and we can write $\textup{Col}_n(K)$ and $\widetilde{\textup{Col}}_n(K)$.

We will briefly describe the $\mathbb{Z}$-module isomorphism $\phi:\widetilde{\textup{Col}}_n(D)\rightarrow \textup{Col}_n(D)$ between Dehn and Fox $n$-colorings given in \cite{carter2014three}.  (Note that all addition below is computed modulo $n$.) Let $\widetilde{C}$ be a Dehn $n$-coloring, and say $a\in A$. Define $\phi(\widetilde{C})(a) = \widetilde{C}(f) + \widetilde{C}(f')$ where $f, f'\in F$ are faces of $D$ separated by arc $a$. The Dehn coloring conditions ensure that the resulting labeling is well-defined and yields a Fox $n$-coloring. The process of obtaining a Dehn coloring from a Fox coloring is more complicated.

Let $C$ be a Fox $n$-coloring. Begin by defining $\phi^{-1}(C)(u)=0$ for the unbounded face $u\in F$. For any internal face $f\in F$, choose a path from $u$ to $f$  that intersects the diagram transversely and does not interact with crossings. Say the sequence of faces and arcs intersecting the path is given by $$u=f_1,a_1, f_2, a_2, f_3,a_3, \ldots, a_{m-1}, f_m=f.$$ The value $\phi^{-1}(C)(f)$ is inductively defined as one traverses this path. For $1<i\leq m$, set $\phi^{-1}(C)(f_i) = C(a_{i-1})-\phi^{-1}(C)(f_{i-1})$. Carter-Silver-Williams call this process integration along that path, and they prove that integration of Fox colorings is conservative in the sense that the result is independent of path  \cite{carter2014three}. An example of a corresponding pair of Fox and Dehn 5-colorings is shown in Figure \ref{corrcolors}.  

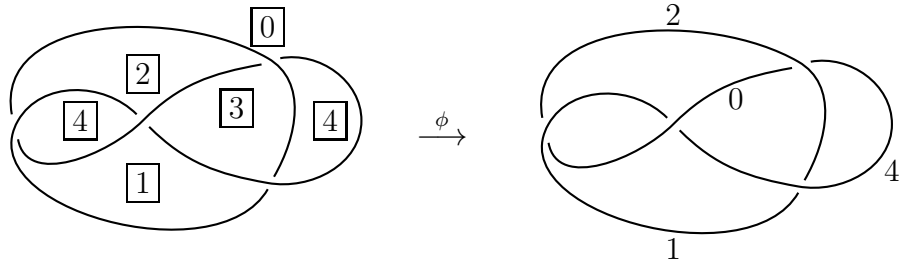
\begin{figure}[h]
\begin{subfigure}[c]{.3\textwidth}
\begin{tikzpicture}[scale=.83]

\node at (4,1.5) {\fbox{0}};
\node at (2, .75) {\fbox{2}};
\node at (3.5,.2) {\fbox{3}};
\node at (5, 0) {\fbox{4}};
\node at (1,0) {\fbox{4}};
\node at (2, -1) {\fbox{1}};

\draw [style=thick] (-.1,.1) to[out=100,in=150] (4,1) to[out=-30, in=60] (4.1,-.9);
\draw [style=thick] (4,-1.1) to[out=240,in=240](0,0) to[out=60, in=135] (1.9,.1);
\draw [style=thick] (2.1,-.1) to[out=-45, in=170] (4,-1) to[out=-10, in=270] (5.5,0) to[out=90, in=10] (4.2,1);
\draw [style=thick] (3.9,.9) to[out=190, in=45] (2,0) to[out=225, in=280] (0,-.3);
\end{tikzpicture}\end{subfigure}
\hspace{.35in}$\stackrel{\phi}{\longrightarrow}$\hspace{.1in}
\begin{subfigure}[c]{.3\textwidth}
\begin{tikzpicture}[scale=.83]

\node at (2,1.75) {2};
\node at (5.5, -.75) {4};
\node at (3,.4) {0};
\node at (2, -2) {1};

\draw [style=thick] (-.1,.1) to[out=100,in=150] (4,1) to[out=-30, in=60] (4.1,-.9);
\draw [style=thick] (4,-1.1) to[out=240,in=240](0,0) to[out=60, in=135] (1.9,.1);
\draw [style=thick] (2.1,-.1) to[out=-45, in=170] (4,-1) to[out=-10, in=270] (5.5,0) to[out=90, in=10] (4.2,1);
\draw [style=thick] (3.9,.9) to[out=190, in=45] (2,0) to[out=225, in=280] (0,-.3);
\end{tikzpicture}
\end{subfigure}
\caption{Corresponding Dehn and Fox colorings of a figure-eight knot diagram} \label{corrcolors}
\end{figure}

Every diagram has at least $n$ distinct $n$-colorings which we call trivial. The trivial Fox $n$-colorings are the $n$ constant maps. Recall every knot diagram has a unique checkerboard shading such that the outer face is white (i.e. unshaded). Trivial Dehn colorings assign the value 0 to every unshaded face and some fixed value to the remaining faces. One can easily verify that $\phi$ maps the trivial Dehn coloring with values $0$ and $l$ to the trivial Fox coloring that labels every strand with $l$. Trivial Fox and Dehn $n$-colorings form submodules of $\textup{Col}_n(D)$ and $\widetilde{\textup{Col}}_n(D)$, and the quotients of the $n$-coloring modules by trivial colorings are referred to as the based $n$-coloring modules. Summarizing, we have the following theorem from \cite{carter2014three}. 

\begin{theorem}[Carter-Silver-Williams]
The map $\phi$ is an isomorphism of the Fox and Dehn $n$-coloring modules and induces an isomorphism on the based Fox and Dehn $n$-coloring modules.
\end{theorem}

The fundamental group of the knot complement, typically denoted $\pi_1(K)$, has two well-known presentations: the Wirtinger and Dehn presentations.  We will only need a few properties of these presentations in our exposition which we recall here. For more extensive information on Wirtinger and Dehn presentations, see for instance \cite{rolfsen1976knots}. 

Both Wirtinger and Dehn presentations are given in terms of generators and relations based on a diagram $D$.  The Wirtinger presentation has generators in one-to-one correspondence with the set $A$ with one relation for each crossing. In the Dehn presentation, the generating set corresponds to the set $F$ (the unbounded face gives a trivial element) with one relation for each crossing. In the following discussion, we conflate the notions of arcs and faces and their corresponding generators.

Finally, the following fact is needed in the proof of Theorem \ref{mapcomm}. If some path from the unbounded face $u$ to a face $f$ passes through the following sequence of arcs and faces:
$u=f_1,a_1, f_2, a_2, f_3,a_3, \ldots, a_{m-1}, f_m=f$, then in $\pi_1(K)$ we have $f = a_1^{\pm1}a_2^{\pm1}\cdots a_{m-1}^{\pm1}$. The $\pm$ signs come from the fact that the Wirtinger presentation depends on an orientation of $D$ while the Dehn presentation does not.

 Let $D_{2n} = \{\alpha, s: \alpha^n = 1=s^2, s\alpha s = \alpha^{-1}\}$ be the standard presentation for the dihedral group with $2n$ elements. Corresponding to a Fox $n$-coloring $C$, there is a homomorphism $\rho_C: \pi_1(K) \rightarrow D_{2n}$ given by $\rho_C(a) = s\alpha^{C(a)}$ where $a\in A$ is a Wirtinger generator of $\pi_1(K)$. We now get the following theorem due to Fox (c.f. \cite{przytycki19983} for a proof).

\begin{theorem}[Fox]\label{reps}
The relation $C\leftrightarrow \rho_C$ is a bijection between the set of all Fox $n$-colorings and the set of maps from $\pi_1(K)$ into $D_{2n}$ sendng Wirtinger generators to reflections.  
\end{theorem}

Let $\widetilde{C}$ be an $n$-Dehn coloring. Again, note that the checkerboard shading of $D$ where the unbounded region is unshaded is unique. Define the map $\widetilde{\rho}_{\widetilde{C}}: \pi_1(K) \rightarrow D_{2n}$ as follows where $f\in F$ is a Dehn generator of $\pi_1(K)$. $$\widetilde{\rho}_{\widetilde{C}}(f) = \begin{cases} s\alpha^{\widetilde{C}(f)} &\mbox{if $f$ is a shaded face and} \\ 
\alpha^{\widetilde{C}(f)} & \mbox{if $f$ is an unshaded face} \end{cases}$$  This definition of $\widetilde{\rho}_{\widetilde{C}}$ agrees with the map $\phi$ as shown in the following theorem.

\begin{theorem} \label{mapcomm}
The map $\widetilde{\rho}_{\widetilde{C}}$ is a homomorphism. Moreover, $\widetilde{\rho}_{\widetilde{C}} = \rho_{\phi(\widetilde{C})}$.
\end{theorem}
\begin{proof}
Let $\widetilde{C}$ be a Dehn $n$-coloring with $\phi(\widetilde{C})=C$, and let $f\in F$ be a Dehn generator of $\pi_1(K)$. Say $u=f_1,a_1, f_2, a_2, f_3,a_3, \ldots, a_{m-1}, f_m=f$ is the ordered list of faces and arcs along some path from the unbounded face $u$ to $f$. Then $f$ can be written in terms of Wirtinger generators as $f = a_1^{\pm1}a_2^{\pm1}\cdots a_{m-1}^{\pm1}$. By definition of $\phi$ we know $\widetilde{C}(f) = C(a_{m-1})-\widetilde{C}(f_{m-1}) = \pm C(a_1) \mp C(a_2) \pm \cdots + C(a_{m-1})$. 

In the dihedral group, reflections are their own inverses, so $(s\alpha^j)^{-1} = s\alpha^j$ for all $j$. It follows that 
\begin{eqnarray*}
\rho_C(f) &=& \rho_C(a_1^{\pm1}a_2^{\pm1}\cdots a_{m-1}^{\pm1}) \\
&=& \rho_C(a_1)\rho_C(a_2) \cdots \rho_C(a_{m-1})\\
&=& s\alpha^{C(a_1)}s\alpha^{C(a_2)} \cdots s\alpha^{C(a_{m-1})}\\
&=& s^k\alpha^{\pm C(a_1) \mp C(a_2) \pm \cdots + C(a_{m-1})}\\
&=&s^k\alpha^{\widetilde{C}(f)}.
\end{eqnarray*}

From this calculation, we see that $\rho_C(f) = s\alpha^{\widetilde{C}(f)}$ or $\rho_C(f) = \alpha^{\widetilde{C}(f)}$ depending on the parity of $k$. While $k$ may depend on the choice of path from $u$ to $f$, the parity of $k$ does not. Furthermore, we see that the parity of $k$ encodes checkerboard shading information. Indeed, if $k$ is odd then $f$ is a shaded face, and if $k$ is even then $f$ is unshaded. We conclude that $\rho_C(f) = \widetilde{\rho}_{\widetilde{C}}(f)$.
\end{proof}

\section{Invariants obtained from coloring data}
Given a diagram $D$, let $\textup{Col}(D)$ be the the free abelian group generated by the set of arcs modulo the Fox coloring relations taken over $\mathbb{Z}$ rather than $\mathbb{Z}/n$. Similarly, define $\widetilde{\textup{Col}}(D)$ to be the free abelian group generated by the set of faces of modulo the Dehn coloring relations taken over $\mathbb{Z}$.  We call $\textup{Col}(D)$ and $\widetilde{\textup{Col}}(D)$  respectively the Fox and Dehn coloring modules. The aim of this section is to study these modules. As with $n$-coloring modules, one can show that the Fox and Dehn coloring modules are isomorphic via the map $\phi$.

\begin{lemma} \label{coloringgroups}
By computing sums over $\mathbb{Z}$ rather than $\mathbb{Z}/n$, the map $\phi$ induces a $\mathbb{Z}$-module isomorphism of $\textup{Col}(D)$ and $\widetilde{\textup{Col}}(D)$. 
\end{lemma}

Coloring modules encode a wealth of information. Indeed, combining Lemma \ref{coloringgroups} with results about Fox coloring we get the following theorem. (These results are due to Fox, but one can find proofs compatible with our discussion in \cite{przytycki19983}.)

\begin{theorem} \label{coloringfacts}
Let $D$  be a diagram of a knot $K$. Then we have the following.
\begin{enumerate}
\item{Coloring modules are invariant under Reidemeister moves, so we can write them as $\textup{Col}(K) \cong \widetilde{\textup{Col}}(K)$.}
\item{The coloring module reduced modulo $n$ is the $n$-coloring module. In other words, $\textup{Col}(K) \otimes  \mathbb{Z}/n \cong\textup{Col}_n(K) \cong \widetilde{\textup{Col}}_n(K) \cong \widetilde{\textup{Col}}(K) \otimes \mathbb{Z}/n$.}
\item{Coloring modules encode important topological information. In particular, $\textup{Col}(K) \cong \widetilde{\textup{Col}}(K) \cong H_1((M_D)^{(2)}, \mathbb{Z}) \oplus \mathbb{Z}$ where $H_1((M_D)^{(2)}, \mathbb{Z})$ is the first homology group of the cyclic branched double cover of $S^3$ branched along $K$.}
\item{Coloring modules have rank 1. In other words, $H_1((M_D)^{(2)}, \mathbb{Z})$ is torsion.}
\item The order of $H_1((M_D)^{(2)}, \mathbb{Z})$ is the absolute value of the Alexander polynomial at -1.
\end{enumerate}
\end{theorem}

The theorem above demonstrates that the torsion part $H_1((M_D)^{(2)}, \mathbb{Z})$ of the coloring module, denoted here by $T(\textup{Col}(K)) \cong T(\widetilde{\textup{Col}}(K))$ to emphasize the coloring connection, is of primary interest since it determines for which $n$ a knot is $n$-colorable. The order of $T(\textup{Col}(K))$ is known as the determinant of the knot and is denoted by $\textup{det}(K)$. The following result about $n$-colorability, due to Fox, is implied by part (2) of Theorem \ref{coloringfacts}. 

\begin{corollary}[Fox]
A knot is $n$-colorable if and only if $\textup{gcd}(\textup{det}(K), n) \neq 1$.
\end{corollary} 

In Section 5, we give a method for computing $\textup{det}(K)$ graph theoretically. To count the number of $n$-colorings of a knot, one needs to look not only at the quantity $\textup{det}(K)$ but at the entire structure of $T(\textup{Col}(K))$. We can determine this structure using relations matrices for the Fox and Dehn coloring modules. 

Given a reduced $k$-crossing diagram $D$ for a knot $K$, a relations matrix $M'(D)$ for $\textup{Col}(K)$ has one row for each crossing and one column for each arc. Similarly, a relations matrix $\widetilde{M}'(D)$ for $\widetilde{\textup{Col}}(K)$ has one row for each crossing and one column for each bounded face. By Euler characteristic arguments, we see that $M'(D)$ is a $k \times k$ square matrix while $\widetilde{M}'(D)$ has dimensions $k\times (k+1)$. The precise matrices obtained depend on an ordering of the crossings, arcs, and faces of the diagram. Figure \ref{knotandmatrices} shows an example.

\begin{figure}[h]
\raisebox{35pt}{$ \left( \begin{array}{cccc}
-1 & 0 & -1 & 2 \\
0 & -1 & 2 & -1 \\
2 & -1 & -1 & 0 \\
-1 & 2 & 0 & -1 \end{array} \right)$}
\hspace{.1in}\large{\raisebox{40pt}{$\stackrel{M'(D)}{\longleftarrow}$}}
\begin{tikzpicture}[scale=.65]

\node at (2,1.75) {\small{$a_1$}};
\node at (5.6, -.75) {\small{$a_2$}};
\node at (3.2,.4) {\small{$a_3$}};
\node at (2, -2) {\small{$a_4$}};

\node at (2, .75) {\small{$f_1$}};
\node at (3.5,-.2) {\small{$f_3$}};
\node at (5, 0) {\small{$f_2$}};
\node at (1,0) {\small{$f_5$}};
\node at (2, -1) {\small{$f_4$}};

\draw [style=thick] (-.1,.1) to[out=100,in=150] (4,1) to[out=-30, in=60] (4.1,-.9);
\draw [style=thick] (4,-1.1) to[out=240,in=240](0,0) to[out=60, in=135] (1.9,.1);
\draw [style=thick] (2.1,-.1) to[out=-45, in=170] (4,-1) to[out=-10, in=270] (5.5,0) to[out=90, in=10] (4.2,1);
\draw [style=thick] (3.9,.9) to[out=190, in=45] (2,0) to[out=225, in=280] (0,-.3);
\end{tikzpicture}
\large{\raisebox{40pt}{$\stackrel{\widetilde{M}'(D)}{\longrightarrow}$}}\hspace{.1in}
\scalebox{.8}{\raisebox{50pt}{$ \left( \begin{array}{ccccc}
1 & 0 & 0 & -1 & 1 \\
-1 & 0 & -1 & 1 & 1 \\
-1 & 1 & 1 & 0 & 0 \\
0 & -1 & 1 & 1 & 0 \end{array} \right)$}}

\caption{Presentation matrices for the Fox and Dehn coloring modules of the figure-eight knot} \label{knotandmatrices}
\end{figure}

Since we are most interested in $T(\textup{Col}(K)) \cong T(\widetilde{\textup{Col}}(K))$, it is useful to find appropriate submatrices of $M'(D)$ and $\widetilde{M}'(D)$ that describe only these submodules. For Fox coloring, the following result explains how to obtain this submatrix. For a proof, see \cite{livingston1993knot}.

\begin{theorem}\label{Foxsub}
Removing any one generator and one relation from the presentation of $\textup{Col}(K)$ does not impact the group structure. Furthermore, removing any one generator gives a presentation for $T(\textup{Col}(K))$.
\end{theorem}

Using the isomorphism $\phi$, we can determine a relations submatrix for $T(\widetilde{\textup{Col}}(K))$. It is important to emphasize that these two presentations are for the same module. We will see, however, that the Dehn coloring module presentation has important advantages in our context.

\begin{theorem}
Removing a single generator corresponding to any face adjacent to the unbounded face of $D$ gives a presentation for $T(\widetilde{\textup{Col}}(K))$.
\end{theorem}
\begin{proof}
Let $f$ be a face of $D$ adjacent to the unbounded face and $a$ be an arc separating them. By Theorem \ref{Foxsub}, the submodule $M\leq \textup{Col}(D)$ generated by all arcs except $a$ modulo crossing relations is $T(\textup{Col}(K))$. By definition of the map $\phi$, it follows that $\phi^{-1}(M) = T(\widetilde{\textup{Col}}(K))$, and $\phi^{-1}(M)$ is generated by all faces except $f$ modulo crossing relations. 
\end{proof}

Denote by $M(D)$ and $\widetilde{M}(D)$ respectively the submatrices (both square) presenting the Fox and Dehn coloring torsion submodules. We obtain $M(D)$  by removing any one row and any one column of $M'(D)$ while $\widetilde{M}(D)$ is obtained by removing one column from $\widetilde{M}'(D)$ that corresponds to a face adjacent to the unbounded face. These are two different presentation matrices for the full torsion submodule of the coloring module. To determine the structure of this submodule, we can compute the Smith normal form of either of these matrices. We review this process in preparation for the upcoming sections. 

Let $X$ be a matrix with integer entries. The {\it Smith normal form} of $X$, denoted $\textup{SNF}(X)$, is the diagonal matrix with entries $(s_1, \ldots, s_k)$ such that $s_i | s_{i+1}$ for all $1\leq i<k$ where $$s_i = \frac{d_i}{d_{i-1}}$$  and $d_i$ is the gcd of all $i\times i$ minors of $X$. 

The values $s_i$ are called invariant factors, and the values $d_i$ are called determinental divisors of the matrix \cite{newman1997smith}. 

If $X$ is a presentation matrix for an abelian group $G$, then the invariant factors completely determine the structure of $G$. In particular, $$G \cong \mathbb{Z}/s_1 \oplus \mathbb{Z}/s_2 \oplus \cdots \oplus \mathbb{Z}/s_k.$$ Any two presentation matrices for the same group will yield the same list of invariant factors up to insertion or deletion of 1's. 

Ge-Jablan-Kauffman-Lopes have a useful theorem which interprets SNF data specifically as it relates to colorings \cite{ge2012equivalence}. We restate it using our notation.
\begin{theorem}
Let $n$ be some positive integer. Let $s_1, \ldots, s_k$ be the list of invariant factors in the Smith normal form for $M(D)$ (or $\widetilde{M}(D)$). Then the number of $n$-colorings of $D$ is $$ |\textup{Col}_n(D)| = \lvert\widetilde{\textup{Col}}_n(D)\rvert = n \cdot \prod_{i=1}^k \gcd(s_i, n).$$
\end{theorem}

We conclude this section by computing $\textup{SNF}(\widetilde{M}(D))$ from the example in Figure \ref{knotandmatrices}. 
Removing the first column from $\widetilde{M}'(D)$ gives the presentation matrix for $T(\widetilde{\textup{Col}}(K))$ shown below. Using the algorithm described above, we obtain the diagonal matrix shown. Hence we conclude that $\widetilde{\textup{Col}}(K) = \mathbb{Z}/5 \oplus \mathbb{Z}$ and the determinant of the knot is 5. If $\textup{gcd}(5,n)=1$, then $T(\widetilde{\textup{Col}}_{n}(K)) = \mathbb{Z}/n$. For any multiple $5l$ of 5, we have $T(\widetilde{\textup{Col}}_{5l}(K)) = \mathbb{Z}/{5} \oplus \mathbb{Z}/{5l}$.
$$\widetilde{M}(D)=  \left( \begin{matrix} 
 0 & 0 & -1 & 1 \\
 0 & -1 & 1 & 1 \\
 1 & 1 & 0 & 0 \\
 -1 & 1 & 1 & 0 \end{matrix} \right) \hspace{.25in}  \rightsquigarrow \hspace{.25in} \textup{SNF}(\widetilde{M}(D))=\left( \begin{matrix}
1 & 0 & 0 & 0 \\
 0 & 1 & 0 & 0 \\
   0 & 0 & 1 & 0 \\
 0 & 0 & 0 & 5 \\
 \end{matrix} \right).$$ 
 
In the case of the figure-eight knot, the Smith normal form does not give any more information than the determinant. When the determinant is composite with repeated factors, though, Smith normal form is a finer invariant. Note that we could just as easily have  computed Smith normal form using the presentation matrix for the entire coloring module. In that case, one additional invariant factor of 0 would appear. The graph theoretic computations in the next two sections work by focusing on the torsion part, which is why we have emphasized it here.

\section{Encoding matrix data via graphs}

We have not yet seen any substantive differences between Fox and Dehn coloring. The remainder of the paper shows how Dehn coloring more naturally lends itself to interpreting coloring data using graphs. We begin with a review of pertinent graph theory terms and a discussion about how one can interpret any matrix as a weighted bipartite graph.

A {\it bipartite graph} $\Gamma = (V, W, E)$ has vertex sets $V$ and $W$ and edge set $E$ where all edges have one endpoint in $V$ and one in $W$. When $E$ contains every possible edge between $V$ and $W$, we say that $\Gamma$ is a {\it complete bipartite graph}. Complete bipartite graphs on $n+m$ vertices are sometimes denoted $K_{n,m}$. A {\it complete graph} on $n$ vertices, denoted by $K_n$, has $n$ vertices with every possible edge between them.

A graph is called {\it planar} if it can be properly embedded in the plane. A {\it plane} graph is a planar graph with a chosen embedding. A graph $H$ is called a {\it minor} of a graph $\Gamma$ if $H$ comes from deleting vertices and deleting or contracting edges of $\Gamma$. A well-known result in graph theory states that a finite graph is planar if and only if it does not contain $K_5$ or $K_{3,3}$ as a minor \cite{wagner1937eigenschaft}.

A {\it dimer} is an edge in a graph. A {\it dimer covering} is a subset $m$ of $E$ such that each vertex in $\Gamma$ is an endpoint of exactly one edge in $m$. Dimer coverings are also called perfect matchings. Let $\mathcal{M}$ be the set of all dimer coverings of $\Gamma$; note that $\mathcal{M} = \emptyset$ whenever $|V| \neq |W|$. 

Given a $p \times q$ matrix $X = (x_{ij})$, we construct a weighted complete bipartite graph $\Gamma_X= (V, W, E, \mu_X)$ with vertex sets $V = \{v_1, \ldots, v_p\}$ and $W = \{w_1, \ldots, w_q\}$ and weight function $\mu_X(v_iw_j) = x_{ij}$. According to this definition, the graph $\Gamma_X$ corresponding to any matrix with $p,q\geq 3$ would contain an induced copy of the complete bipartite graph $K_{3,3}$ and hence would be nonplanar. 

For our purposes, it is sufficient to restrict to edges with nonzero weight. Even with the convention of omitting weight zero edges, many matrices do not produce planar graphs. We are particularly interested in the graphs corresponding to $M(D)$ and $\widetilde{M}(D)$. 

\begin{remark}
There exist knot diagrams $D$ for which $\Gamma_{M(D)}$ is not planar.
\end{remark}

 As an example, consider the diagram for the knot $9_{42}$ shown in Figure \ref{942example}. The figure also shows a corresponding presentation matrix $M(D)$  for $T(\textup{Col}(9_{42}))$, the graph $\Gamma_{M(D)}$, and a subgraph of $\Gamma_{M(D)}$ that is a subdivision of $K_{3,3}$. Since $\Gamma_{M(D)}$ contains a a subdivision of $K_{3,3}$, it is not a planar graph. 

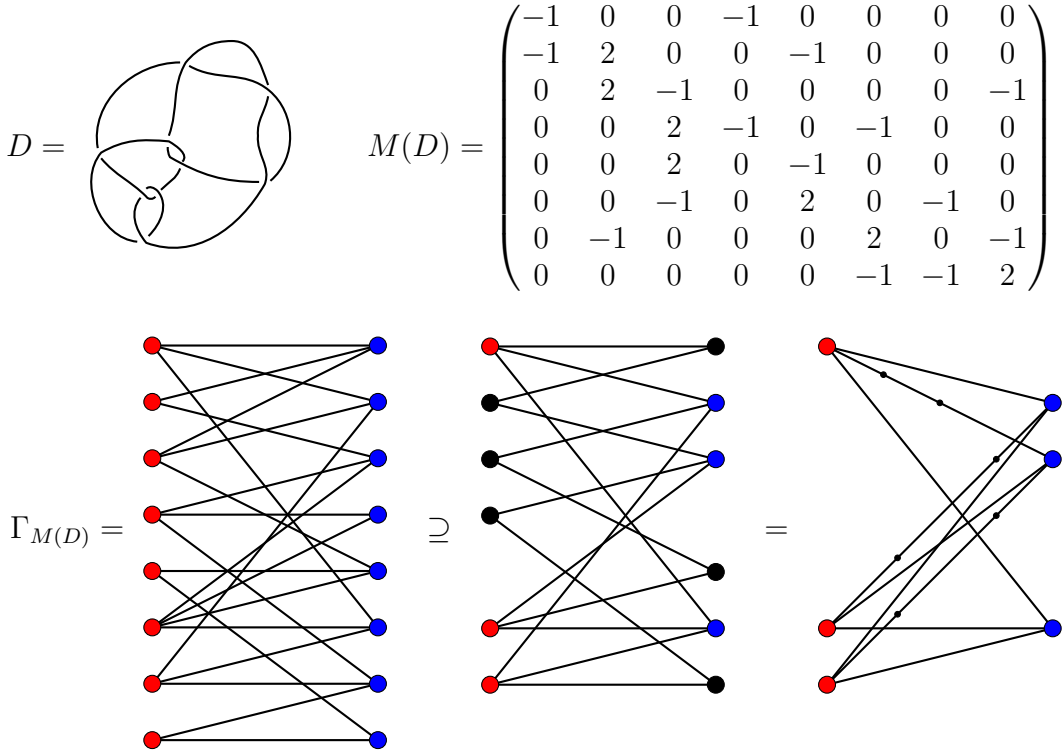
\begin{figure}[h]
\begin{center}
\raisebox{40pt}{$D =$} 
\begin{tikzpicture}[scale=.15]
\draw[style=thick] (0, 0.5) to[out=70, in=240] (1.5, 7) to [out=50, in=180](5.6, 8.9) to[out=360,in=110] (8.8,5);
\draw[style=thick] (8.8,4) to [out=260, in=95] (7.9, .5) to [out=265, in=70](8.5,-3.8) to [out=240, in=340](-2,-9) to [out=130, in=230] (-2.5, -5);
\draw[style=thick](-2.8, -8.9) to [out=175, in=230] (-6, -1) to [out=30, in=175] (0,0) to [out=330, in=45] (1.2, -1.8);
\draw[style=thick] (-1.8, -8.4) to [out=45, in= 310] (-1, -4.4) to [out=135, in=45] (-2, -4.2);
\draw[style=thick] (-1.2, -4.9) to [out=240, in=300] (-2, -4.9) to [out= 135, in= 310](-6, -1.5);
\draw[style=thick] (-.7, -4.3) to [out=45, in=250] (1,-2.4);
\draw[style=thick] (-6.3, -.6) to[out=100, in=180] (1, 7);
\draw[style=thick] (1.8, 6.7) to [out=330, in=140](8.9, 4.5) to [out=320, in=40] (9, -3.5);
\draw[style=thick] (8, -3.6) to [out=180, in=330] (.1,-1.4) to [out=117, in=270](-.1, -.5);
\end{tikzpicture}
\hspace{5mm} \raisebox{40pt}{
$ M(D) = \left(
\begin{matrix}
 -1 & 0 & 0 & -1 & 0 & 0 & 0 & 0 \\ 
 -1 & 2 & 0 & 0 & -1 & 0 & 0 & 0 \\ 
 0 & 2 & -1 & 0 & 0 & 0 & 0 & -1 \\ 
 0 & 0 & 2 & -1 & 0 & -1 & 0 & 0 \\ 
 0 & 0 & 2 & 0 & -1 & 0 & 0 & 0 \\ 
 0 & 0 & -1 & 0 & 2 & 0 & -1 & 0 \\ 
 0 & -1 & 0 & 0 & 0 & 2 & 0 & -1 \\ 
 0 & 0 & 0 & 0 & 0 & -1 & -1 & 2
\end{matrix} \right) $ }

\vspace{.2in}

 \raisebox{80pt}{$\Gamma_{M(D)}=$}\hspace{.1in}\begin{tikzpicture}[scale=.75]

\draw[color=black, style=thick] (0,1) -- (4,1);
\draw[color=black, style=thick] (0,1) -- (4,2);
\draw[color=black, style=thick] (0,2) -- (4,2);
\draw[color=black, style=thick] (0,2) -- (4,3);
\draw[color=black, style=thick] (0,2) -- (4,7);
\draw[color=black, style=thick] (0,3) -- (4,3);
\draw[color=black, style=thick] (0,3) -- (4,4);
\draw[color=black, style=thick] (0,3) -- (4,5);
\draw[color=black, style=thick] (0,3) -- (4,6);
\draw[color=black, style=thick] (0,4) -- (4,1);
\draw[color=black, style=thick] (0,4) -- (4,4);
\draw[color=black, style=thick] (0,5) -- (4,2);
\draw[color=black, style=thick] (0,5) -- (4,5);
\draw[color=black, style=thick] (0,5) -- (4,6);
\draw[color=black, style=thick] (0,6) -- (4,4);
\draw[color=black, style=thick] (0,6) -- (4,7);
\draw[color=black, style=thick] (0,6) -- (4,8);
\draw[color=black, style=thick] (0,7) -- (4,6);
\draw[color=black, style=thick] (0,7) -- (4,8);
\draw[color=black, style=thick] (0,8) -- (4,3);
\draw[color=black, style=thick] (0,8) -- (4,7);
\draw[color=black, style=thick] (0,8) -- (4,8);

\draw[fill=red,radius=.15] (0,1)circle;
\draw[fill=red,radius=.15] (0,2)circle;
\draw[fill=red,radius=.15] (0,3)circle;
\draw[fill=red,radius=.15] (0,4)circle;
\draw[fill=red,radius=.15] (0,5)circle;
\draw[fill=red,radius=.15] (0,6)circle;
\draw[fill=red,radius=.15] (0,7)circle;
\draw[fill=red,radius=.15] (0,8)circle;

\draw[fill=blue,radius=.15] (4,1)circle;
\draw[fill=blue,radius=.15] (4,2)circle;
\draw[fill=blue,radius=.15] (4,3)circle;
\draw[fill=blue,radius=.15] (4,4)circle;
\draw[fill=blue,radius=.15] (4,5)circle;
\draw[fill=blue,radius=.15] (4,6)circle;
\draw[fill=blue,radius=.15] (4,7)circle;
\draw[fill=blue,radius=.15] (4,8)circle;

\end{tikzpicture} \hspace{.1in} \raisebox{80pt}{$\supseteq$}\hspace{.1in}
\raisebox{21pt}{\begin{tikzpicture}[scale=.75]

\draw[color=black, style=thick] (0,2) -- (4,2);
\draw[color=black, style=thick] (0,2) -- (4,3);
\draw[color=black, style=thick] (0,2) -- (4,7);
\draw[color=black, style=thick] (0,3) -- (4,3);
\draw[color=black, style=thick] (0,3) -- (4,4);

\draw[color=black, style=thick] (0,3) -- (4,6);

\draw[color=black, style=thick] (0,5) -- (4,2);

\draw[color=black, style=thick] (0,5) -- (4,6);
\draw[color=black, style=thick] (0,6) -- (4,4);
\draw[color=black, style=thick] (0,6) -- (4,7);

\draw[color=black, style=thick] (0,7) -- (4,6);
\draw[color=black, style=thick] (0,7) -- (4,8);
\draw[color=black, style=thick] (0,8) -- (4,3);
\draw[color=black, style=thick] (0,8) -- (4,7);
\draw[color=black, style=thick] (0,8) -- (4,8);

\draw[fill=red,radius=.15] (0,2)circle;
\draw[fill=red,radius=.15] (0,3)circle;

\draw[fill=black,radius=.15] (0,5)circle;
\draw[fill=black,radius=.15] (0,6)circle;
\draw[fill=black,radius=.15] (0,7)circle;
\draw[fill=red,radius=.15] (0,8)circle;

\draw[fill=black,radius=.15] (4,2)circle;
\draw[fill=blue,radius=.15] (4,3)circle;
\draw[fill=black,radius=.15] (4,4)circle;

\draw[fill=blue,radius=.15] (4,6)circle;
\draw[fill=blue,radius=.15] (4,7)circle;
\draw[fill=black,radius=.15] (4,8)circle;

\end{tikzpicture}} \hspace{.1in} \raisebox{80pt}{$=$}\hspace{.1in}
\raisebox{21pt}{\begin{tikzpicture}[scale=.75]

\draw[color=black, style=thick] (0,2) -- (4,6);
\draw[color=black, style=thick] (0,2) -- (4,3);
\draw[color=black, style=thick] (0,2) -- (4,7);
\draw[color=black, style=thick] (0,3) -- (4,3);
\draw[color=black, style=thick] (0,3) -- (4,7);

\draw[color=black, style=thick] (0,3) -- (4,6);

\draw[color=black, style=thick] (0,8) -- (4,6);
\draw[color=black, style=thick] (0,8) -- (4,3);
\draw[color=black, style=thick] (0,8) -- (4,7);

\draw[fill=red,radius=.15] (0,2)circle;
\draw[fill=red,radius=.15] (0,3)circle;

\draw[fill=black,radius=.05] (3,5)circle;
\draw[fill=black,radius=.05] (3,6)circle;

\draw[fill=black,radius=.05] (1,7.5)circle;
\draw[fill=black,radius=.05] (2,7)circle;

\draw[fill=red,radius=.15] (0,8)circle;

\draw[fill=black,radius=.05] (1.25,4.25)circle;
\draw[fill=black,radius=.05] (1.25,3.25)circle;
\draw[fill=blue,radius=.15] (4,3)circle;

\draw[fill=blue,radius=.15] (4,6)circle;
\draw[fill=blue,radius=.15] (4,7)circle;

\end{tikzpicture}}
\end{center}
\caption{A nonplanar graph corresponding to $T(\textup{Col}(9_{42}))$}
\label{942example}
\end{figure}

The graph $\Gamma_{\widetilde{M}(D)}$, on the other hand, is always planar with a particularly nice plane embedding reflecting the structure of the knot. In fact, this graph is a weighting of the balanced overlaid Tait graph which has appeared in several recent papers \cite{cohen2012determinant, CDR, cohen2012kauffman}.

A balanced overlaid Tait (BOT) graph $\Gamma_D = (V, W, E)$ corresponding to a knot diagram $D$ which is used in \cite{CDR} (see also \cite{cohen2012determinant, cohen2012kauffman}) is constructed as follows.
\begin{itemize}
\item{The vertex set $V$ is the set of crossings of $D$.}
\item{The vertex set $W$ is the set of all but one bounded face of $D$. The omitted face is always chosen to be adjacent to the unbounded face.}
\item{Given vertices $x\in V$ and $y\in W$, the edge $xy\in E$ if and only if the crossing $x$ is incident to face $y$.}
\end{itemize} 

Planarity of the BOT graph follows directly from its construction.This graph is called the balanced overlaid Tait graph since it comes from superimposing the two Tait graphs corresponding to the diagram and introducing a third vertex set where the two graphs meet. This means that the BOT graph is in fact tripartite, but we will not use this here.  Whenever we consider BOT graphs, we fix the plane embedding that comes from the knot diagram.

\begin{lemma}\label{matdimer}
Consider a knot diagram $D$ and presentation matrix $\widetilde{M}(D)$ for $T(\widetilde{\textup{Col}}(D))$ that comes from removing face generator $f$. The graph $\Gamma_{\widetilde{M}(D)}$ is the weighted BOT graph for $D$ coming from omitting face $f$. 
\end{lemma}
\begin{proof}
The matrix $\widetilde{M}(D)$ is obtained from $\widetilde{M}'(D)$ by removing one column corresponding to the bounded face $f$ of $D$ adjacent to the unbounded region. The graph $\Gamma_{\widetilde{M}(D)}$ has one vertex for each crossing, and one vertex for each bounded face except $f$. An edge of $\Gamma_{\widetilde{M}(D)}$ has nonzero weight exactly when a face and crossing meet. Hence $\Gamma_{\widetilde{M}(D)}$ (with edges of weight 0 omitted) is exactly the BOT graph for $D$ that comes from omitting face $f$.
\end{proof}

\begin{figure}
\begin{tikzpicture}[scale=.7]

\draw [color=gray] (-.1,.1) to[out=100,in=150] (4,1) to[out=-30, in=60] (4.1,-.9);
\draw [ color=gray] (4,-1.1) to[out=240,in=240](0,0) to[out=50, in=145] (1.9,.1);
\draw [ color=gray] (2.1,-.1) to[out=-45, in=170] (4,-1) to[out=-10, in=270] (5.5,0) to[out=90, in=10] (4.2,1);
\draw [ color=gray] (3.9,.9) to[out=190, in=45] (2,0) to[out=225, in=320] (0,-.2);

\draw[style=thick, double] (-.1,0)--(.75, -1);
\draw[style=ultra thick] (0,0)--(2,0);
\draw[style=ultra thick] (4,1)--(5.15,0);
\draw[style=ultra thick] (2,0)--(.75, -1);
\draw[style=thick, double] (4,-1)--(5.15,0);
\draw[style=ultra thick] (4,-1)--(.75, -1);
\draw[style=ultra thick] (4,1)--(3.5, 0);
\draw[style=ultra thick] (4,-1)--(3.5,0);
\draw[style=thick, double] (2,0)--(3.5,0);

\draw[fill=black,radius=.15] (0,0)circle;
\draw[fill=black, radius=.15] (4,-1)circle;
\draw[fill=black, radius=.15] (2,0)circle;
\draw[fill=black, radius=.15] (4,1)circle;

\draw[fill=white, radius=.15] (.75,0)circle;
\draw[fill=white, radius=.15] (3.5,0)circle;
\draw[fill=white, radius=.15] (.75,-1)circle;
\draw[fill=white, radius=.15] (5.15,0)circle;

\node at (1.8,.75) {\Huge{*}};

\end{tikzpicture}
\caption{A BOT graph for a figure-eight diagram}\label{fig8dimer}
\end{figure}
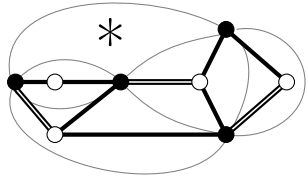

For an example of a BOT graph, see Figure \ref{fig8dimer}. In the figure, the knot diagram is in gray, the omitted bounded face is marked with an asterisk, and doubled lines indicate edges with weight $-1$. Solid edges have weight $+1$.

Now that we have explained a method for translating a matrix to a graph, the goal is to graph theoretically compute matrix data using dimer coverings and  Kasteleyn theory. This method was originally developed to work in the opposite direction -- to obtain graph data via a matrix \cite{kasteleyn1967graph, temperley1961dimer}. Indeed, if one has a planar graph, the number of perfect matchings can be computed by taking a certain determinant. We briefly recall selected concepts from Kasteleyn theory below. For more details, see \cite{kenyon3129lectures, kuperberg1994symmetries, kuperberg1998exploration, kuperberg2002kasteleyn}.

Let $\Gamma = (V, W, E, \mu)$ be a simple, weighted bipartite graph with $V = \{v_1, \ldots, v_p\}$ and $W = \{w_1, \ldots, w_q\}$.  Then the weight matrix for $\Gamma$, denoted by $M_{\Gamma}$, is a $p\times q$ matrix with entries $x_{ij}$ given by $$x_{ij} =  \begin{cases}  \mu(v_iw_j) &\mbox{if $v_iw_j\in E$ and} \\ 
0 & \mbox{otherwise}. \end{cases}.$$
Up to transposition and rearrangement of rows and columns, $M_{\Gamma_{\widetilde{M}(D)}} = \widetilde{M}(D)$. In particular, up to sign, these two matrices have the same determinant and the same Smith normal form.

The following partition function $Z(\Gamma)$ which computes a sum over the set $\mathcal{M}$ of dimer coverings (or perfect matchings) of $\Gamma$ is of particular interest in statistical physics.

$$Z(\Gamma)=\sum_{m \, \in \, \mathcal{M}} \prod_{e \, \in \, m} \mu(e)$$ 
Note that if $\mu\equiv 1$, then $Z(\Gamma)$ counts the dimer coverings of $\Gamma$. Also note that, in terms of $M_\Gamma$, we have $\textup{perm}(M_\Gamma) = \pm Z(\Gamma)$ where $\textup{perm}(M_\Gamma)$ is the permanent or unsigned determinant of the matrix $M_\Gamma$.  When $\Gamma$ is planar (and in certain other cases we won't consider here), one can modify the weighting $\mu$ so that $Z(\Gamma)$ can be computed via a determinant. This modification is called a Kasteleyn weighting.

Let $\Gamma =(V,W,E)$ be a bipartite plane graph. A {\it Kasteleyn weighting} $\epsilon: E \rightarrow \{ \pm 1\}$ has the property that each bounded face with 0 (mod 4) edges has an odd number of $-1$ assignments and each bounded face with 2 (mod 4) edges has an even number of $-1$ assignments. 

\begin{theorem}[Kasteleyn] \label{Kweight}
Every bipartite plane graph  has a Kasteleyn weighting. 
\end{theorem}

One can prove Theorem \ref{Kweight} constructively. For a proof, see \cite{CDR}.  Finally, we state Kasteleyn's theorem which is crucial for our constructions in the next section. For a proof of this theorem, see \cite{kuperberg1998exploration}.

\begin{theorem}[Kasteleyn] \label{KasteleynsTheorem}
Let  $\Gamma=(V, W, E, \mu)$ be a weighted bipartite plane graph, and let $\Gamma'=(V, W, E, \mu\cdot\epsilon)$ where $\epsilon$ is any Kasteleyn weighting of $\Gamma$. Then $$Z(\Gamma') = \sum_{m \, \in \, \mathcal{M}} \prod_{e \, \in \, m} \epsilon(e)\mu(e) = \textup{perm}(M(\Gamma')) = \pm \textup{det}(M(\Gamma))$$ or equivalently $$Z(\Gamma) = \textup{perm}(M(\Gamma)) =  \pm \textup{det}(M(\Gamma')).$$
\end{theorem}

\section{Dimers and the SNF}

The coloring data for a knot can be completely obtained from the Dehn or Fox coloring matrix. While this is computationally convenient, the matrix may obscure important structural information about the knot. Our goal in computing matrix quantities using the BOT graph is to more directly relate coloring data to properties of the diagram. 

The approach of interpreting a determinant via a partition function is not new. It has been used in the study of Ozsv\'{a}th-Szab\'{o} Knot-Floer
homology theory and by the last author and coauthors in studying twisted Alexander polynomials  \cite{CDR, lowrance2008knot, ozsvath2003heegaard}. Applying Theorems \ref{Kweight} and \ref{KasteleynsTheorem}, the following algorithm shows how the determinant of a knot can be computed via a partition function on a BOT graph.

\begin{algorithm} \label{detalg}
Given a diagram $D$ for a knot $K$, one can compute $\textup{det}(K)$ as follows.
\begin{enumerate}
\item{Construct a weighted BOT graph $\Gamma_D$ with weighting $\mu$ coming from the Dehn coloring conditions at each crossing. Call this the {\it Dehn weighting}.  Figure \ref{Dehnweight} shows a local weighting where double edges indicate a -1 weight.   
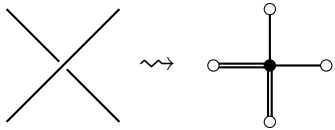
\begin{figure}[h]
\begin{tikzpicture}
[scale=.5]

\draw[style=thick] (0,0) to[out=45,in=-135] (3,3);
\draw[style=thick] (3,0) to[out=135, in=-45] (1.6,1.4);
\draw[style=thick] (0,3) to[out=-45, in=135] (1.4,1.6);

\node at (4,1.5) {\large{$\leadsto$}};

\draw[style=thick, double] (7,0) -- (7,1.5);
\draw[style=thick] (7,1.5) -- (7,3);
\draw[style=thick, double] (5.5, 1.5)--(7, 1.5);
\draw[style=thick] (7, 1.5)--(8.5, 1.5);

\draw[radius=.15, fill=black](7,1.5)circle;
\draw[radius=.15, fill=white](5.5,1.5)circle;
\draw[radius=.15, fill=white](8.5,1.5)circle;
\draw[radius=.15, fill=white](7,0)circle;
\draw[radius=.15, fill=white](7,3)circle;
%\node at (5.15,.8) {\tiny{$x$}};
%\node at (6.35,.8) {\tiny{$y$}};
%\node at (5.15,-.8) {\tiny{$z$}};
%\node at (6.35,-.8) {\tiny{$w$}};

\end{tikzpicture}
\caption{A local Dehn weighting}\label{Dehnweight}
\end{figure}}
\item{Fix a Kasteleyn weighting $\epsilon$ on $\Gamma_D$. (For alternating knots, Lemma \ref{DehnKast} shows there is a convenient choice.)}
\item{Compute $\textup{det}(K)$ via the following formula.  $$\textup{det}(K) = Z(\Gamma'_D)= \sum_{m\in \mathcal{M}} \prod_{e\in m} \epsilon(e)\mu(e)$$}
\end{enumerate}
\end{algorithm}

Alternating diagrams have a particularly natural Kasteleyn weighting as demonstrated by the following lemma. Even if one has a nonalternating knot,  we can apply this result to get a Kasteleyn weighting by considering an alternating knot with the same shadow.

\begin{lemma}\label{DehnKast}
The Dehn weighting on a BOT graph for an alternating diagram is a Kasteleyn weighting.
\end{lemma}

\begin{proof}
Every bounded face of the BOT graph is a square, so in order to have a Kasteleyn weighting each bounded face should have one or three edges with weight -1. If a diagram is alternating, each bounded face in the BOT graph comes from an alternating pair of crossings as shown below. 

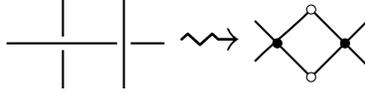
\begin{figure}[h]
\begin{tikzpicture}
[scale=.6]
%\node at (-.5,.5) {\small{$b$}};
%\node at (-.5,-.5) {\small{$a$}};
%\node at (.9,.5) {\small{$d$}};
%\node at (.9,-.5) {\small{$c$}};
%\node at (2.25,.5) {\small{$f$}};
%\node at (2.25,-.5) {\small{$e$}};

\draw[style=thick] (-1,0) to[out=0,in=180] (1.45,0);
\draw[style=thick] (1.75,0) to[out=0, in=180] (2.5,0);
\draw[style=thick] (1.6,1) to[out=270, in=90] (1.6, -1);
\draw[style=thick] (.25,1) to[out=270, in=90] (.25, .15);
\draw[style=thick] (.25,-.15) to[out=270, in=90] (.25, -1);

\node at (3.5,0) {\huge{$\leadsto$}};

\draw[style=thick] (4.5,-.4) to[out=45, in=-135] (5.75, .75);
\draw[style=thick] (5.75,-.75) to[out=45, in=-135] (7, .5);
\draw[style=thick] (7,-.5) to[out=135, in=-45] (5.75, .75);
\draw[style=thick] (5.75,-.75) to[out=135, in=-45] (4.5, .5);

\draw[radius=.1, fill=black](5,0)circle;
\draw[radius=.1, fill=black](6.5,0)circle;
\draw[radius=.1, fill=white](5.75,.75)circle;
\draw[radius=.1, fill=white](5.75,-.75)circle;

%\node at (5.15,.8) {\tiny{$x$}};
%\node at (6.35,.8) {\tiny{$y$}};
%\node at (5.15,-.8) {\tiny{$z$}};
%\node at (6.35,-.8) {\tiny{$w$}};

\end{tikzpicture}
\end{figure}

Each edge of the square face shown will have a weight of either 1 or -1. The Dehn coloring relations require that the two lefthand edges have the same weight while the two righthand edges have opposite weights. With those constraints, the square face will always have exactly one or three edges with weight -1. 
\end{proof}

Lemma \ref{DehnKast} implies that the number of perfect matchings of a BOT graph for an alternating diagram is the determinant of the associated knot. As we discuss afterwards, this is a classical result about the number of spanning trees of the Tait graph of an alternating diagram reworded in the language of the BOT graph  \cite{crowell1959nonalternating}.

\begin{corollary}\label{dehnkastcor}
Let $\Gamma_D$ be a BOT graph for an alternating diagram $D$ of a knot $K$. Then $\textup{det}(K)$ is the number of dimer coverings on $\Gamma_D$.
\end{corollary}

\begin{proof}
Let $\widetilde{M}(D)$ be a matrix for $T(\widetilde{\textup{Col}}(D))$. Lemma \ref{matdimer} showed that $\Gamma_{\widetilde{M}(D)} = \Gamma_D$ is a BOT graph for $D$. Lemma \ref{DehnKast} showed that the Dehn weighting $\mu:E \rightarrow \{\pm1\}$ is a Kasteleyn weighting. Applying Kasteleyn's Theorem \ref{KasteleynsTheorem}, we see that $$\textup{det}(K) = |\textup{det}(\widetilde{M}(D))| = Z\left(\Gamma'_D\right) = \sum_{m \, \in \, \mathcal{M}} \prod_{e \, \in \, m} (\mu(e))^2 = \sum_{m \, \in \, \mathcal{M}} \prod_{e \, \in \, m} 1 = |\mathcal{M}|.$$
\end{proof}

In Formal Knot Theory \cite{kauffman2006formal}, Kauffman discusses the one-to-one correspondence between marked states of the link diagram and rooted spanning trees of the Tait graph (paired with rooted spanning trees of the dual to the Tait graph). There is an obvious one-to-one correspondence between Kauffman's marked states and elements of $\mathcal{M}$ which is discussed in \cite{CDR}. The reader can also see \cite{cohen2012determinant} for more on correspondences between these objects. Figure \ref{correspondence} shows a corresponding marked diagram, perfect matching, and spanning tree.
\begin{figure}[h]
\begin{tikzpicture}[scale=.6]

\draw [thick] (-.1,.1) to[out=100,in=150] (4,1) to[out=-30, in=60] (4.1,-.9);
\draw [thick] (4,-1.1) to[out=240,in=240](0,0) to[out=50, in=145] (1.9,.1);
\draw [thick] (2.1,-.1) to[out=-45, in=170] (4,-1) to[out=-10, in=270] (5.5,0) to[out=90, in=10] (4.2,1);
\draw [thick] (3.9,.9) to[out=190, in=45] (2,0) to[out=225, in=320] (0,-.2);
\filldraw[style=thick, color=black](4.05,1)--(4.2,.82)--(4.4, 1)--(4,1);
\filldraw[style=thick, color=black](4.05,-1)--(3.8,-1.3)--(3.75,-.95)--(3.95,-1);
\filldraw[style=thick, color=black](2,0)--(2.2,.2)--(2.2, -.2)--(2,0);
\filldraw[style=thick, color=black](-.07,-.17)--(.15,.12)--(.15, -.27)--(-.07,-.17);

\node at (1.8,.75) {\Large{*}};

\end{tikzpicture} \hspace{.17in}
\raisebox{25pt}{\begin{tikzpicture}[scale=.6]
\draw[color=gray] (-.1,0)--(.75, -1);
\draw[style=ultra thick] (0,0)--(.8,0);
\draw[color=gray] (2,0)--(.8,0);
\draw[style=ultra thick] (4,1)--(5.15,0);
\draw[color=gray] (2,0)--(.75, -1);
\draw[color=gray,] (4,-1)--(5.15,0);
\draw[style=ultra thick] (4,-1)--(.75, -1);
\draw[color=gray] (4,1)--(3.5, 0);
\draw[color=gray] (4,-1)--(3.5,0);
\draw[style=ultra thick] (2,0)--(3.5,0);

\draw[fill=black,radius=.15] (0,0)circle;
\draw[fill=black, radius=.15] (4,-1)circle;
\draw[fill=black, radius=.15] (2,0)circle;
\draw[fill=black, radius=.15] (4,1)circle;

\draw[fill=white, radius=.15] (.75,0)circle;
\draw[fill=white, radius=.15] (3.5,0)circle;
\draw[fill=white, radius=.15] (.75,-1)circle;
\draw[fill=white, radius=.15] (5.15,0)circle;
\end{tikzpicture}}
\hspace{.07in}
\begin{tikzpicture}[scale=.6]

\draw [color=gray] (-.1,.1) to[out=100,in=150] (4,1) to[out=-30, in=60] (4.1,-.9);
\draw [ color=gray] (4,-1.1) to[out=240,in=240](0,0) to[out=50, in=145] (1.9,.1);
\draw [ color=gray] (2.1,-.1) to[out=-45, in=170] (4,-1) to[out=-10, in=270] (5.5,0) to[out=90, in=10] (4.2,1);
\draw [ color=gray] (3.9,.9) to[out=190, in=45] (2,0) to[out=225, in=320] (0,-.2);

\draw[style=ultra thick] (1.8,.9) to[out=20, in=165](4,.97) to[out=-15, in=90] (5.15,0);
\draw[style=ultra thick](5.15,0) to[out=-90,in=20] (4,-1) to[out=200, in=-10] (.75, -1);
\draw[fill=black, radius=.15] (.75,-1)circle;
\draw[fill=black, radius=.15] (5.15,0)circle;

\draw[style=ultra thick] (-1,0)--(3.5,0);

\draw[fill=white, radius=.15] (3.5,0)circle;
\draw[fill=white, radius=.15] (1,0)circle;

\draw[fill=white, white, radius=.25] (-1.05,0)circle;
\draw[fill=white, white, radius=.25] (2,.72)circle;

\node at (-1,-.15) {\Large{*}};
\node at (1.8,.75) {\Large{*}};

\end{tikzpicture}
\caption{A corresponding marked diagram, perfect matching, and spanning tree}\label{correspondence}
\end{figure}
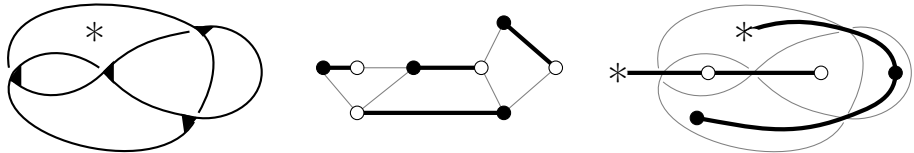

So far we have encoded the Dehn torsion matrix as a weighted graph and computed its determinant via a partition function. The final goal is to compute the Smith normal form of the knot from this perspective. The graph theoretic analog of a matrix minor is an induced subgraph. Removing rows and columns of a matrix corresponds to removing vertices in the associated graph.  We give a general algorithm for computing $\textup{SNF}(K)$ and then focus on cases where this algorithm can be simplified.

\begin{algorithm}\label{SNFalg}
Consider a $k$-crossing diagram $D$ for a knot $K$. Compute $\textup{SNF}(K)$ as follows.
\begin{enumerate}
\item{ Construct a Dehn weighted BOT graph $\Gamma_D = (V, W, E, \mu)$ as in Algorithm \ref{detalg}. }
\item{Beginning with $i=k$, compute $d_i$ using the following steps. For each pair of subsets $V'\subset V$ and $W'\subset W$ with $|V'|=|W'| = i$, consider the induced subgraph $\Gamma_D^{V'W'}$ of the BOT graph spanned by $V'\cup W'$. Fix a Kasteleyn weighting $\epsilon_{V'W'}$ on $\Gamma_D^{V'W'}$. Using the following formula, compute $d_{V'W'}$ where $\mathcal{M}_{V'W'}$ is the set of all dimer coverings of $\Gamma_D^{V'W'}$.   $$d_{V'W'} = \sum_{m\in \mathcal{M}_{V'W'}} \prod_{e\in m} \epsilon_{V'W'}(e)\mu(e)$$ Then $d_i$ can be computed as  $d_i = \gcd \{ d_{V'W'}: |V'|=|W'|=i \}.$ If $d_i = 1$, then $d_j = 1$ for all $1\leq j<i$. If $d_i \neq 1$, then decrement $i$ by one and repeat the above computation.}  
\item{Finally, we have $s_i = \frac{d_i}{d_{i-1}}$, and $(s_1, s_2, \ldots, s_k)$ determines $\textup{SNF}(K)$.}
\end{enumerate}
\end{algorithm}

In Algorithm \ref{SNFalg}, for each subset $V'\cup W'$ of the vertex set of $\Gamma_D$, we must determine a new Kasteleyn weighting on the induced subgraph $\Gamma_D^{V'W'}$. In general this is tedious, but there are certain cases where computations are easier. For example, when one removes a pair of adjacent vertices (and all edges incident to them) from $\Gamma_D$, any Kasteleyn weighting on $\Gamma_D$ restricted to the remaining subgraph is still a Kasteleyn weighting.

\begin{lemma}
Let $\Gamma = (V,W, E)$ be a plane bipartite graph equipped with a Kasteleyn weighting $\epsilon$. Let $v\in V$ and $w\in W$ be a pair of adjacent vertices in $\Gamma$. Then $\epsilon$ restricts to a Kasteleyn weighting on the subgraph of $\Gamma$ obtained by removing $v$ and $w$.
\end{lemma}

\begin{proof}

If the edge $vw$ or one of the vertices $v$ or $w$ is incident to the unbounded face, then removing all edges connected to $v$ and $w$ will not introduce new bounded faces. Hence $\epsilon$ will restrict to a Kasteleyn weighting of the subgraph induced by all vertices except $v$ and $w$. 

Now assume that neither the edge $vw$ nor the vertices $v$ or $w$ are incident to the unbounded face. Say that there are $l$ faces enumerated $f_1, \ldots, f_l$ that contain $v$ or $w$ (or both) as part of their boundaries. Each of these is bounded and square. Let $y_i$ be the number of $-1$ weights on edges of $f_i$ incident to $v$ or $w$. Let $z_i$ be the number of $-1$ weights on edges of $f_i$ not incident to $v$ or $w$. 

Since $\epsilon$ is a Kasteleyn weighting and all faces are square, it follows that $y_i+z_i$ is odd for all $i$. Consider $\sum_{i=1}^l (y_i+z_i) = \sum_{i=1}^l y_i + \sum_{i=1}^l z_i$. The sum $\sum_{i=1}^l y_i$ is always even since each edge adjacent to $v$ or $w$ is part of two faces in the list $f_1, \ldots, f_l$. Since $y_i+z_i$ is odd for all $i$, the entire sum has the same parity as $l$. We conclude that  $\sum_{i=1}^l z_i$ has the same parity as $l$.   

Consider the subgraph of $\Gamma$ obtained by removing $v$, $w$, and any edges incident to $v$ and $w$. Any bounded face of $\Gamma$  not containing $v$ or $w$ persists in this subgraph, and thus $\epsilon$ satisfies the Kasteleyn condition on each of these faces. There is one new face created when $v$ and $w$ are removed which has $2l-2$ edges since removing $v$ and $w$ deletes 3 edges from 2 faces and 2 edges from the other $l-2$ faces with $v$ or $w$ in their boundaries. If $l$ is odd, then $2l-2 \equiv 0 (\textup{mod } 4)$. If $l$ is even, then $2l-2 \equiv 2 (\textup{mod } 4)$. The number of -1 weights around the new face is $\sum_{i=1}^l z_i$ which matches the parity of $l$. This means that the Kasteleyn condition is satisfied on the only face that has been modified by deleting $v$ and $w$. Hence $\epsilon$ restricts to a Kasteleyn weighting on this induced subgraph.
\end{proof}

There is another observation that can be useful when computing minors graph theoretically. This gives conditions under which a weighted edge can be added to a graph while maintaining the Kasteleyn weighting. As a corollary we can compute certain minors without modifying Kasteleyn weightings.

\begin{lemma} \label{oliverslemma}
Let $\Gamma = (V,W, E)$ be a plane bipartite graph equipped with a Kasteleyn weighting $\epsilon$. Say $v\in V$ and $w\in W$ with $vw\notin E$. If the edge $vw$ can be added to $\Gamma$ while maintaining planarity of its embedding, then $\epsilon$ extends to a Kasteleyn weighting on $\Gamma\cup vw$. 
\end{lemma}
\begin{proof}
The addition of $vw$ to $\Gamma$ will split some face $f$ into two smaller faces. If $f$ has $2l$ edges, then $l>2$ (since we should not obtain a multigraph by adding $vw$) and $vw$ splits $f$ into faces of size $2m$ and $2(l-m+1)$ for some $1<m<l$.

If $l$ is even, then $2l\equiv 0 (\textup{mod }4)$, and $\epsilon$ assigns an odd number of $-1$ weights to the edges of $f$. Note that  $l$ even implies that $m$ and $l-m+1$ have opposite parity. Therefore, when $vw$ is added, one new face will have an even number of $-1$ weights and the other will have an odd number. 

If $l$ is odd, then $2l \equiv 2 (\textup{mod }4)$, and $\epsilon$ assigns an odd number of $-1$ weights to the edges of $F$. For $l$ odd, note that $m$ and $l-m+1$ have the same parity. When $vw$ is added to $\Gamma$, either both new faces will have an odd number of $-1$ weights or both will have an even number. 

In any of the above cases, it follows that either both new faces or neither of the new faces will have the correct number of $-1$ signs needed for a Kasteleyn weighting. If no additional $-1$ weights are needed, set $\epsilon(vw)=1$. Otherwise, setting $\epsilon(vw)=-1$ will yield a Kasteleyn weighting on $\Gamma\cup vw$. 
\end{proof}

\begin{corollary} \label{removal}
Let $\Gamma=(V,W,E)$ be a planar bipartite graph with Kasteleyn weighting $\epsilon$ and nonadjacent vertices $v\in V$, $w\in W$ lying on a common face. Then $\epsilon$ restricts  to a Kasteleyn weighting on $\Gamma-\{v,w\}$.
\end{corollary}

Using these results, there are shortcuts for computing $\textup{SNF}(K)$. One begins the $\textup{SNF}$ computation by finding the $\textup{det}(K) = d_k$. Already if $d_k$ is prime or has no repeating prime factors, then by the definition of $\textup{SNF}$ we know $s_k = d_k = \textup{det}(K)$ and $s_i = 1$ for all $1\leq i <k$. In this case, the coloring torsion module is cyclic. With $\textup{det}(K)$ in hand, one can determine the possible Smith normal forms and hence the maximum number of iterations of Algorithm \ref{SNFalg} needed. Even when $\textup{det}(K)$ is composite with repeated factors, there are certain cases when we can determine that the coloring torsion module is cyclic without removing every possible pair of vertices. 

An edge $e$ in a graph $\Gamma$ is called a {\it forcing edge} if there is exactly one dimer covering of $\Gamma$ containing it \cite{ZZ}. More generally a subset $S$ of the vertex set of $\Gamma$ is called a forcing set if $S$ is contained in exactly one dimer covering of $\Gamma$. Using this language, the following theorem gives a graph theoretic sufficient condition for determining when the coloring torsion module is cyclic.

\begin{theorem}
If a BOT graph for a diagram $D$ of $K$ has a forcing edge, then $\textup{SNF}(K)$ has invariant factors $(1, \ldots, 1, \det{K})$ and the coloring torsion module is cyclic.
\end{theorem}

\begin{corollary}
If a BOT graph for an $k$-crossing diagram $D$ of $K$ has a forcing set of size $2m$, then $\textup{SNF}(K)$ has at most $k-m$ nontrivial factors.
\end{corollary}

\begin{example}
Consider the alternating diagram for the knot $8_8$ and its BOT graph shown in Figure \ref{BOTex}. One can verify that there are 25 dimer coverings. Thus $\textup{det}(K) = d_8 = 25$. This means that the $\textup{SNF}(8_8)=(1,1,1,1,1,1, 1, 25)$ or $\textup{SNF}(8_8)=(1, 1,1,1,1, 1,5, 5)$. 

The third part of Figure \ref{BOTex} shows the subgraph induced on all but two vertices. By Corollary \ref{removal}, the number of perfect matchings in the subgraph is the corresponding $7\times 7$ minor. All forced edges are red; there are only two possible matchings. This means $d_7=1$ or $d_7=2$. Since $d_8$ is odd and $d_7$ must divide $d_8$, we know that $d_7=1$.  We conclude that $\textup{SNF}(8_8)=(1,1,1,1,1,1, 1, 25)$, and the coloring torsion module is cyclic.
\begin{figure}[h]
\begin{tikzpicture}[scale=.7]
\draw [style=ultra thick] (0,0) to[out=0, in=135](1, -1) to[out=-45, in=-135] (1.8, -1.1);
\draw [style=ultra thick] (2.1, -.8) to[out=60, in=-60] (2,2) to[out=120, in=-120] (1.9,4);
\draw[style=ultra thick] (2.1, 4.3) to[out=45, in=45] (3,3) to [out=225, in=30] (2.1,2.2);
\draw[style=ultra thick] (1.8, 2.1) to[out=190, in=-20] (0,2) to[out=170, in=70] (0,4);
\draw[style=ultra thick](.1, 4.4) to[out=70, in=135] (2.1,4.1) to[out=-45, in=115] (2.8, 3.1);
\draw[style=ultra thick] (3, 2.8) to[out=-65, in=0] (2,-1) to[out=180, in=45] (1.1, -.9);
\draw[style=ultra thick] (.9, -1.2) to[out=-135, in=-120] (-.2,0) to[out=70, in=-120] (-.1,1.8);
\draw[style=ultra thick] (.1, 2.1) to[out=60, in=0] (0,4.3) to[out=180, in=170] (-.4, 0);
\end{tikzpicture}\hspace{.4in}
\begin{tikzpicture}[scale=.7]
\draw [color=gray] (0,0) to[out=0, in=135](1, -1) to[out=-45, in=-135] (1.8, -1.1);
\draw [color=gray] (2.1, -.8) to[out=60, in=-60] (2,2) to[out=120, in=-120] (1.9,4);
\draw[color=gray] (2.1, 4.3) to[out=45, in=45] (3,3) to [out=225, in=30] (2.1,2.2);
\draw[color=gray] (1.8, 2.1) to[out=190, in=-20] (0,2) to[out=170, in=70] (0,4);
\draw[color=gray](.1, 4.4) to[out=70, in=135] (2.1,4.1) to[out=-45, in=115] (2.8, 3.1);
\draw[color=gray] (3, 2.8) to[out=-65, in=0] (2,-1) to[out=180, in=45] (1.1, -.9);
\draw[color=gray] (.9, -1.2) to[out=-135, in=-120] (-.2,0) to[out=70, in=-120] (-.1,1.8);
\draw[color=gray] (.1, 2.1) to[out=60, in=0] (0,4.3) to[out=180, in=170] (-.4, 0);

\draw[style= ultra thick] (2.05,4.1)--(2.75,3.75);
\draw[style= ultra thick] (2.05,4.1)--(2.25,3);
\draw[style= ultra thick] (2.05,4.1)--(1,3);
\draw[style= ultra thick] (0.1,4.3)--(1,3);
\draw[style= ultra thick] (0.1,4.3)--(-1,2) ;
\draw[style= ultra thick] (0.1,4.3)--(0,3);
\draw[style= ultra thick] (0,2)--(1,3);
\draw[style= ultra thick] (0,2)--(-1,2);
\draw[style= ultra thick] (0,2)--(0,3);
\draw[style= ultra thick] (0,2)--(1,1);
\draw[style= ultra thick] (-.2,0)--(-1,2);
\draw[style= ultra thick]  (2,-1)--(1,1);
\draw[style= ultra thick] (1, -1)--(1,1);
\draw[style= ultra thick] (1.9,2.1)--(1,1);
\draw[style= ultra thick] (-.2,0)--(1,1);
\draw[style= ultra thick] (1.9,2.1)--(1,3);
\draw[style= ultra thick] (2.95,2.95)--(2.75,3.75);
\draw[style= ultra thick] (2.95,2.95)--(2.25,3);
\draw[style= ultra thick](1.9,2.1)--(2.25,3);
\draw[style= ultra thick](-.2,0)--(.25,-.5);
\draw[style= ultra thick](1,-1)--(.25,-.5);
\draw[style= ultra thick](1,-1)--(1.45, -1);
\draw[style= ultra thick](2,-1)--(1.45, -1);

\draw[fill=black,radius=.15] (2.05,4.1)circle;
\draw[fill=black,radius=.15] (1.9,2.1)circle;
\draw[fill=black,radius=.15] (2,-1)circle;
\draw[fill=black,radius=.15] (1, -1)circle;
\draw[fill=black,radius=.15] (-.2,0)circle;
\draw[fill=black,radius=.15] (0,2)circle;
\draw[fill=black,radius=.15] (0.1,4.3)circle;
\draw[fill=black,radius=.15] (2.95,2.95)circle;

\draw[fill=white,radius=.13] (-1,2)circle;
\draw[fill=white,radius=.13] (0,3)circle;
\draw[fill=white,radius=.13] (1,3)circle;
\draw[fill=white,radius=.13] (1,1)circle;
\draw[fill=white,radius=.13] (2.25,3)circle;
\draw[fill=white,radius=.13] (2.75,3.75)circle;
\draw[fill=white,radius=.13] (1.45,-1)circle;
\draw[fill=white,radius=.13] (.25,-.5)circle;

\node at (3, 1) {\Large{$*$}};

\end{tikzpicture}\hspace{.5in}
\raisebox{7pt}{\begin{tikzpicture}[scale=.7]

%\draw[style= ultra thick] (2.05,4.1)--(2.75,3.75);
%\draw[style= ultra thick] (2.05,4.1)--(2.25,3);
%\draw[style= ultra thick] (2.05,4.1)--(1,3);
\draw[style= ultra thick] (0.1,4.3)--(1,3);
\draw[style= ultra thick] (0.1,4.3)--(-1,2) ;
\draw[style= ultra thick] (0.1,4.3)--(0,3);
\draw[style= ultra thick] (0,2)--(1,3);
\draw[style= ultra thick] (0,2)--(-1,2);
\draw[style= ultra thick] (0,2)--(0,3);
%\draw[style= ultra thick] (0,2)--(1,1);
\draw[style= ultra thick, red] (-.2,0)--(-1,2);
%\draw[style= ultra thick]  (2,-1)--(1,1);
%\draw[style= ultra thick] (1, -1)--(1,1);
%\draw[style= ultra thick] (1.9,2.1)--(1,1);
%\draw[style= ultra thick] (-.2,0)--(1,1);
\draw[style= ultra thick] (1.9,2.1)--(1,3);
\draw[style= ultra thick, red] (2.95,2.95)--(2.75,3.75);
\draw[style= ultra thick] (2.95,2.95)--(2.25,3);
\draw[style= ultra thick, red](1.9,2.1)--(2.25,3);
\draw[style= ultra thick](-.2,0)--(.25,-.5);
\draw[style= ultra thick, red](1,-1)--(.25,-.5);
\draw[style= ultra thick](1,-1)--(1.45, -1);
\draw[style= ultra thick, red](2,-1)--(1.45, -1);

%\draw[fill=black,radius=.15] (2.05,4.1)circle;
\draw[fill=black,radius=.15] (1.9,2.1)circle;
\draw[fill=black,radius=.15] (2,-1)circle;
\draw[fill=black,radius=.15] (1, -1)circle;
\draw[fill=black,radius=.15] (-.2,0)circle;
\draw[fill=black,radius=.15] (0,2)circle;
\draw[fill=black,radius=.15] (0.1,4.3)circle;
\draw[fill=black,radius=.15] (2.95,2.95)circle;

\draw[fill=white,radius=.13] (-1,2)circle;
\draw[fill=white,radius=.13] (0,3)circle;
\draw[fill=white,radius=.13] (1,3)circle;
%\draw[fill=white,radius=.13] (1,1)circle;
\draw[fill=white,radius=.13] (2.25,3)circle;
\draw[fill=white,radius=.13] (2.75,3.75)circle;
\draw[fill=white,radius=.13] (1.45,-1)circle;
\draw[fill=white,radius=.13] (.25,-.5)circle;
\end{tikzpicture}}
\caption{A BOT graph and induced subgraph for $8_8$}\label{BOTex}
\end{figure}
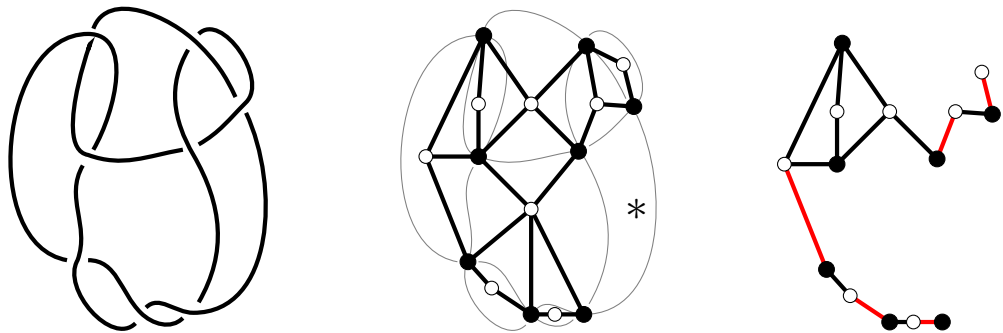
\end{example}

\bibliographystyle{plain}
\bibliography{DimerRefs-arxiv}

\end{document}